\documentclass[11pt]{amsart}

\usepackage{amsfonts,amssymb,amscd,amsmath,enumerate,verbatim,calc,thumbpdf,mathrsfs, graphicx, multicol, multirow, stmaryrd,hyperref}
\usepackage[all]{xy}

\theoremstyle{plain}
\newtheorem{Theorem}{Theorem}[section]
\newtheorem{Lemma}[Theorem]{Lemma}
\newtheorem{Corollary}[Theorem]{Corollary}
\newtheorem{Proposition}[Theorem]{Proposition}
\newtheorem{intthm}{Theorem}

\theoremstyle{definition}

\newtheorem{Notation}[Theorem]{Notation}
\newtheorem{Remark}[Theorem]{Remark}

\newtheorem{Fact}[Theorem]{Fact}
\newtheorem{Ex}[Theorem]{Example}
\newtheorem{Note}[Theorem]{Note}
\newtheorem{Definition}[Theorem]{Definition}
\newtheorem{Example}[Theorem]{Example}
\newtheorem*{convention*}{Convention}

\def\depth{\mbox{\rm {depth}}}

\newcommand{\Ht}{\operatorname{ht}}	
\newcommand{\del}{\delta}
\newcommand{\mrad}{\operatorname{m-rad}}
\newcommand{\rad}{\operatorname{rad}}

\newcommand{\mht}{\operatorname{m-ht}}

\usepackage[utf8]{inputenc}
\usepackage[T1]{fontenc}
\numberwithin{equation}{Theorem}

\begin{document}

\title[Edge Ideals of Weighted Graphs]{Edge Ideals of Weighted Graphs}
\author{Chelsey Paulsen \and Sean Sather-Wagstaff}

\address{Department of Mathematics 2750\\ North Dakota State University\\PO BOX 6050\\ Fargo, ND 58108-6050\\ USA}

\email{chelsey.paulsen@gmail.com}
\email{sean.sather-wagstaff@ndsu.edu}
\urladdr{http://www.ndsu.edu/pubweb/\~{}ssatherw/}


\dedicatory{Dedicated to Warren Shreve on the eve of his retirement}

\begin{abstract}
We study weighted graphs and their ``edge ideals'' which are ideals in polynomial rings that are defined in terms of the graphs. We provide combinatorial descriptions of m-irreducible decompositions for the edge ideal of a weighted graph in terms of the combinatorics of ``weighted vertex covers''. We use these, for instance, to say when these ideals are m-unmixed. We explicitly describe which weighted cycles, suspensions, and trees are unmixed and which ones are Cohen-Macaulay, and we prove that all weighted complete graphs are Cohen-Macaulay.
\end{abstract}

\maketitle

\section*{Introduction}

\begin{convention*}
Throughout this paper, let $A$ be a  non-zero commutative ring,
and let $R$ denote a polynomial ring $R= A[X_1,\dots,X_d]$.
Let $G=(V, E)$ be a (finite simple) graph with
vertex set $V=\{v_1,\ldots,v_d\}$ and edge set $E$.
An edge between vertices $v_i$ and $v_j$ is denoted $v_iv_j$.
\end{convention*}

In this section, assume that $A$ is a field.

Algebra and combinatorics have a rich history of interaction. In short, one can study combinatorial objects (graphs, posets, simplicial complexes, etc.) through algebraic constructions. In the other direction, one can use these constructions to find interesting examples of ideals and rings, for instance, families of Cohen-Macaulay rings. 
This paper continues in this  tradition.

A relatively new (but well-studied) construction takes the graph $G$ and associates to it the ``edge ideal'' $I(G)$ in the polynomial ring $R$. Much work has been done to relate the combinatorial properties of $G$ to the algebraic properties of $I(G)$, and vice versa. For instance, one can explicitly describe the irreducible decomposition of $I(G)$ in terms of the combinatorial structure of $G$. In particular, one can easily describe when $I(G)$ is unmixed. On the other hand, the Cohen-Macaulay property for $R/I(G)$ is more subtle. Much work in the literature is devoted to providing classes of graphs $G$ such that $R/I(G)$ is Cohen-Macaulay (or not) for instance in \cite{Fr,VT,Vi1}.

In this paper, we introduce and study a version of this construction for weighted graphs;
see Sections~\ref{sec120507a} and~\ref{sec120507b} for background material
on weighted graphs and monomial ideals. We study the irreducible decompositions of these ideals via ``weighted vertex covers'' and characterize when these ideals are unmixed in Section~\ref{sec120508a}. We apply this, for instance, to the situation of weighted cycles (which are almost always mixed) and weighted complete graphs (which are always unmixed) in Section~\ref{sec120507c}. We conclude with Section~\ref{sec120507d} which describes some situations where these weighted graphs are Cohen-Macaulay. 
For instance, we completely characterize the Cohen-Macaulay weighted cycles.

\begin{intthm}\label{intthm120507a}
Consider a weighted $d$-cycle $C^d_{\omega}$.
\begin{enumerate}[\rm(a)]
\item \label{intthm120507a1} If $C^d_{\omega}$ is Cohen-Macaulay, then $d\in\{3,5\}$.
\item \label{intthm120507a2} $C^3_{\omega}$ is always Cohen-Macaulay.
\item \label{intthm120507a3} $C^5_{\omega}$ is Cohen-Macaulay if and only if 
it can be written in the form
$$
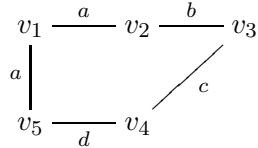
\xymatrix{
v_1 \ar@{-}[d]_a \ar@{-}[r]^a &v_2\ar@{-}[r]^b &v_3 \ar@{-}[ld]^c\\
v_5\ar@{-}[r]_d & v_4
}$$
such that $a\leq b\geq c\leq d\geq e$
\end{enumerate}
\end{intthm}

This result is proved at the  end of Section~\ref{sec120507d}.
In Theorem~\ref{Thrm conditions for trees iff CM}
we also completely characterize the Cohen-Macaulay weighted trees.
This result contains the following:

\begin{intthm}\label{intthm120507b}
If the weighted tree $T_{\omega}$ is Cohen-Macaulay, then 
the underlying tree $T$ is a suspension of a tree, hence 
$T$ is Cohen-Macaulay. Conversely, if $T$ is a Cohen-Macaulay tree, then
there is a weight function $\omega$ such that $T_{\omega}$ is Cohen-Macaulay.
\end{intthm}

Recall that every suspension of a tree is Cohen-Macaulay.
The same is not true for every weighted tree $T_{\omega}$ whose underlying graph
is a suspension of a tree: if $T$ is a suspension of a tree, then the weights on the
``whiskers'' determine whether $T_{\omega}$ is Cohen-Macaulay. 
This is a consequence of Theorem~\ref{thm120607} which characterizes
the Cohen-Macaulay weighted suspension.

As one may expect, we computed a number of examples using Macaulay~2~\cite{M2}
in the process of proving our results, though none of our proofs
depends on these computations.

\section{Weighted Graphs and Weighted Vertex Covers}
\label{sec120507a}

In this section, we introduce weighted vertex covers for weighted graphs and 
describe some of their basic properties for use in subsequent sections.
Recall that $G$ is a graph with vertex set $V= \{v_1, \dots, v_d\}$.

\begin{Definition} \label{Def vertex cover}
A \textbf{vertex cover} of $G$ is a subset $V' \subseteq V$ such that for each edge $v_iv_j$ in $G$ either $v_i \in V'$ or $v_j \in V'$. A vertex cover is \textbf{minimal} if it does not properly contain another vertex cover of $G$.
\end{Definition}

\begin{Definition}\label{Def weight}
A \textbf{weight function} on a graph $G$ is a function $\omega \colon E \rightarrow \mathbb{N}$ that assigns a \textbf{weight} to each edge.\footnote{We assume that
$\mathbb{N}=\{1,2,3,\ldots\}$.}
A \textbf{weighted graph} $G_{\omega}$ is a graph $G$ equipped with a weight function $\omega$.
A weighted graph $G_{\omega}$ where each edge has the same weight is a \textbf{trivially weighted graph}.
\end{Definition}

\begin{Note}
We represent weighted graphs graphically, as in the statement of 
Theorem~\ref{intthm120507a} in the introduction, by labeling
each edge with its weight.
\end{Note}

\begin{Definition} \label{Def weighted vertex cover}
Let $G_{\omega}$ be a weighted graph. A \textbf{weighted vertex cover} of $G_{\omega}$ is an ordered pair $(V', \delta ')$ such that $V'$ is a vertex cover of $G$ and $\delta '\colon V' \rightarrow \mathbb{N}$ is a function such that for each edge $e=v_iv_j\in E$
we have
\begin{enumerate}[(1)]
\item\label{item120507a} $v_i \in V'$ and $\delta'(v_i) \leq \omega (e)$, or
\item\label{item120507b} $v_j \in V'$ and $\delta'(v_j) \leq \omega (e)$.
\end{enumerate}
The number  $\delta'(v_i)$  is the \textbf{weight} of $v_i$. 
When condition~\eqref{item120507a} is satisfied, we write that
the vertex $v_i$ \textbf{covers} the edge $e$, and similarly for 
condition~\eqref{item120507b}.
\end{Definition}

\begin{Notation}
Given a weighted vertex cover $(V', \delta ')$ of a weighted graph $G_{\omega}$,
we sometimes write $V'=\{v_{i_1}^{\delta'(v_{i_1})},\dots,v_{i_k}^{\delta'(v_{i_k})}\}$.
\end{Notation}

\begin{Notation} \label{Def cycle}
For $d\geq 3$, a \textbf{$d$-cycle} is the graph $C^d$  with vertex set $V(C^d)= \{v_1,\dots, v_d\}$ and edge set $E(C^d)=\{v_1v_2,v_2v_3,\dots,v_dv_1\}$.
We denote this graph as $C^d = v_1v_2 \dots v_dv_1$.
\end{Notation}

\begin{Note}
As with weighted graphs, we represent weighted vertex covers graphically.
For instance, the following sketch represents the weighted vertex cover
$\{v_1^a,v_2^b,v_4^d,v_5^a\}$ of the weighted 5-cycle from
Theorem~\ref{intthm120507a} in the introduction: 
$$
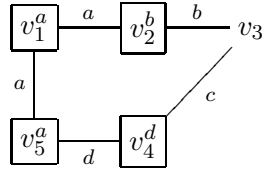
\xymatrix{
*+[F]{v_1^{a}} \ar@{-}[d]_a \ar@{-}[r]^a &*+[F]{v_2^b}\ar@{-}[r]^b &v_3 \ar@{-}[ld]^c\\
*+[F]{v_5^a} \ar@{-}[r]_d & *+[F]{v_4^d}.
}$$
\end{Note}

\begin{Example}
Let $C^5_{\omega}$ denote the following weighted 5-cycle:
$$
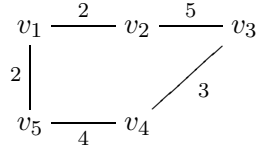
\xymatrix{
v_1 \ar@{-}[d]_2 \ar@{-}[r]^2 &v_2\ar@{-}[r]^5 &v_3 \ar@{-}[ld]^3\\
v_5\ar@{-}[r]_4 & v_4
}$$
Then the first sketch in the following display does not represent a
weighted vertex cover of $C^5_{\omega}$
because the edges $v_1v_2$ and $v_2v_3$ are not covered.
$$
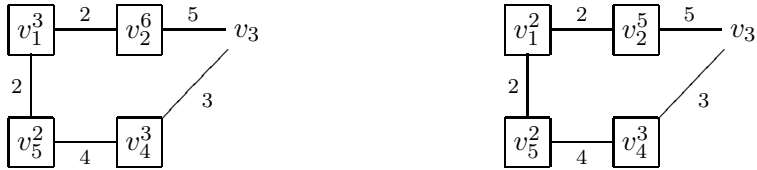
\xymatrix{
*+[F]{v_1^{3}} \ar@{-}[d]_2 \ar@{-}[r]^2 &*+[F]{v_2^6}\ar@{-}[r]^5 &v_3 \ar@{-}[ld]^3\\
*+[F]{v_5^2} \ar@{-}[r]_4 & *+[F]{v_4^3}
}
\qquad\qquad
\qquad\qquad
\xymatrix{
*+[F]{v_1^{2}} \ar@{-}[d]_2 \ar@{-}[r]^2 &*+[F]{v_2^5}\ar@{-}[r]^5 &v_3 \ar@{-}[ld]^3\\
*+[F]{v_5^2} \ar@{-}[r]_4 & *+[F]{v_4^3}.
}$$
The second sketch in this display is a weighted vertex cover of $C^5_{\omega}$.
\end{Example}

We define an ordering of weighted vertex covers next. 

\begin{Definition} Let $G_{\omega}$ be a weighted graph. 
Given two weighted vertex covers $(V',\delta')$ and $(V'',\delta'')$,
write $(V'',\delta'')\leq (V',\delta')$ if $V''\subseteq V'$ and for all $v_i\in V''$ we have $\delta'(i)\leq \delta''(i)$.
A weighted vertex cover $(V',\delta')$ is \textbf{minimal} if there does not exist another weighted vertex cover $(V'',\delta'')$ such that $(V'',\delta'')< (V',\delta')$. We define $|(V',\delta')|=|V'|$.

The graph $G$ is said to be \textbf{unmixed} if all of the minimal vertex covers of $G$ have the same cardinality.  If $G$ is not unmixed then $G$ is \textbf{mixed}. Similarly, a weighted graph $G_{\omega}$ is \textbf{unmixed} if all of the minimal weighted vertex covers of $G_{\omega}$ have the same cardinality.  If $G_{\omega}$ is not unmixed then $G_{\omega}$ is \textbf{mixed}. 
\end{Definition}

\begin{Example}\label{ex120507a}
Let $C^5_{\omega}$ denote the following weighted 5-cycle:
$$
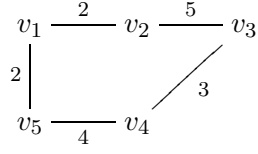
\xymatrix{
v_1 \ar@{-}[d]_2 \ar@{-}[r]^2 &v_2\ar@{-}[r]^5 &v_3 \ar@{-}[ld]^3\\
v_5\ar@{-}[r]_4 & v_4
}$$
Then the first sketch in the following display is
a weighted vertex cover of $C^5_{\omega}$
that is not minimal,
because the weighted vertex $v_5^2$ is not needed:
$$
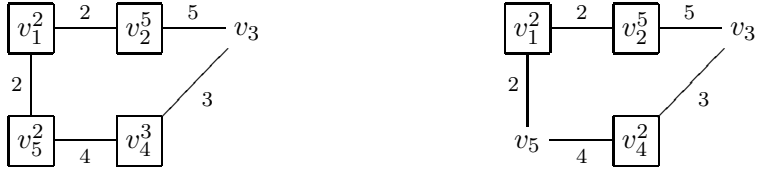
\xymatrix{
*+[F]{v_1^{2}} \ar@{-}[d]_2 \ar@{-}[r]^2 &*+[F]{v_2^5}\ar@{-}[r]^5 &v_3 \ar@{-}[ld]^3\\
*+[F]{v_5^2} \ar@{-}[r]_4 & *+[F]{v_4^3}.
}
\qquad\qquad
\qquad\qquad
\xymatrix{
*+[F]{v_1^{2}} \ar@{-}[d]_2 \ar@{-}[r]^2 &*+[F]{v_2^5}\ar@{-}[r]^5 &v_3 \ar@{-}[ld]^3\\
v_5 \ar@{-}[r]_4 & *+[F]{v_4^2}.
}$$
The second sketch in this display is also
a non-minimal weighted vertex cover of $C^5_{\omega}$
because the weight on $v_4$ can be increased to make
the next weighted vertex cover which is minimal:
$$
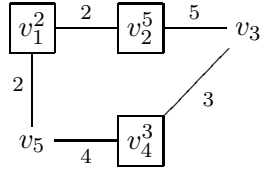
\xymatrix{
*+[F]{v_1^{2}} \ar@{-}[d]_2 \ar@{-}[r]^2 &*+[F]{v_2^5}\ar@{-}[r]^5 &v_3 \ar@{-}[ld]^3\\
v_5 \ar@{-}[r]_4 & *+[F]{v_4^3}.
}$$
Note that this minimal weighted vertex cover can be obtained by removing the
superfluous vertex from the first non-minimal weighted vertex cover.
\end{Example}

The following  results will be useful in the sections that follow.
The first one says that, if the weight on a vertex 
in a weighted vertex cover can be increased without bound,
then that vertex can be removed from the weighted vertex cover.

\begin{Lemma}\label{lem120507a}
Let $G_{\omega}$ be a weighted graph, and assume that, for $j=1,2,\ldots$
we have a weighted vertex cover
$V_j=\{v_1^{a_1},\ldots,v_{n}^{a_n},v_{n+1}^{b_j}\}$ of $G_{\omega}$. 
If $b_1<b_2<\cdots$, then $V'=\{v_1^{a_1},\ldots,v_{n}^{a_n}\}$
is also a weighted vertex cover 
of $G_{\omega}$. 
\end{Lemma}

\begin{proof}
Let $e=v_iv_{n+1}$ be an edge in $G_{\omega}$ with weight $\omega(e)$.
By assumption, there is an index $j$ such that
$b_j>\omega(e)$. Since $V_j$ is a weighted  vertex cover 
of $G_{\omega}$, the edge $e$ must be covered by $v_i$,
that is, we must have $i\leq n$ and $a_i\leq\omega(e)$.
Since this is so for each edge of the form $e=v_iv_{n+1}$,
it follows that every edge of $G_{\omega}$ is covered by
one of the weighted vertices $v_1^{a_1},\ldots,v_{n}^{a_n}$.
In other words, $V'$ is also a weighted vertex cover 
of $G_{\omega}$, as desired.
\end{proof} 

\begin{Proposition} \label{Prop mwvc contained in wvc}
Let $G_{\omega}$ be a weighted graph. Then for every weighted vertex cover $(V',\delta')$ of $G_\omega$ there is a minimal weighted vertex cover $(V''',\delta''')$ 
of $G_{\omega}$ such that $(V''',\del''')\leq(V',\del')$.
\end{Proposition}

\begin{proof}
If $(V',\delta')$ is itself a minimal weighted vertex cover for $G_{\omega}$, then we are done. If $(V',\delta')$ is not minimal, then either there is  a $v_i\in V'$ that can be removed or for some $v_i\in V'$ the function $\delta'(v_i)$ can be increased,
as in Example~\ref{ex120507a}.
In the first case, remove vertices  from $V'$
until the removal of one more vertex creates an ordered pair that is no longer a weighted vertex cover. Notice that this process terminates in finitely many steps because $V'$ is finite. Let us denote this new weighted vertex cover as $(V'',\delta'')$. 
(If no vertices can be removed, then set $(V'',\delta'')=(V',\delta')$.)

Now, if  $(V'',\delta'')$ is a minimal weighted vertex cover for $G_{\omega}$, then we are done. If it is not minimal, then we can increase the weight of at least one of the vertices in $V'$. Increase the  weight of each vertex (in sequence) such that any further increase would cause the ordered pair to not be a weighted vertex cover. 
This process also terminates in finitely many steps because the weight of each vertex of $V''$ can not be increased without bound, by Lemma~\ref{lem120507a}. We will denote this new ordered pair $(V''',\delta''')$. Since no vertices can be removed from $(V''',\delta''')$ and the weight of each $v_i\in V'''$ can not be increased,
the pair $(V''',\delta''')$ is a minimal weighted vertex cover for $G_{\omega}$ such that $(V''',\del''')\leq(V',\del')$. 
\end{proof} 

\begin{Proposition} \label{Prop mvc implies mwvc}
Let $G_{\omega}$ be a weighted graph. Then every minimal vertex cover of the unweighted graph $G$ occurs as a minimal weighted vertex cover of $G_{\omega}$. 
\end{Proposition}

\begin{proof}
Let $V'$ be a minimal vertex cover for $G$. We need to show that $(V',\delta')$ is a minimal weighted vertex cover for $G_{\omega}$ for some $\delta'$. 
For each $v_i\in V'$ define 
$$\delta'(v_i) = \min\{\omega(e)\mid e=v_iv_j\in E \text{ for some } v_j\}.$$
We claim that $(V',\delta')$ is a  weighted vertex cover 
for $G_{\omega}$. Let $e=v_iv_j$ be an edge $G$. If $v_i\in V'$, then by definition of $\delta'$ we have  $\delta'(v_i)\leq \omega(e)$; and if $v_j\in V'$, then 
$\delta'(v_j)\leq \omega(e)$. Hence $(V',\delta')$ is a weighted vertex cover. 

Proposition~\ref{Prop mwvc contained in wvc} provides
a minimal weighted vertex cover $(V''',\delta''')$ 
of $G_{\omega}$ such that $(V''',\del''')\leq(V',\del')$.
Since $V'$ is a minimal vertex cover,  we can not remove any vertices from $V'$. 
Since $V'''$ is a vertex cover for $G$, the condition $V'''\subseteq V'$ implies that $V'''=V'$.
Thus, $V'$ 
occurs as the minimal weighted vertex cover $(V''',\delta''')$. 
\end{proof} 

\begin{Proposition}\label{Prop mixed ei then mixed wei}
If $G$ is mixed, then $G_{\omega}$ is mixed.
\end{Proposition}

\begin{proof}
Assume that 
$G$ is mixed. Then there are  minimal vertex covers $V'$ and $V''$ for $G$ 
such that $|V'|\neq  |V''|$. 
By Proposition~\ref{Prop mvc implies mwvc},   we have  functions $\delta'\colon V'\rightarrow \mathbb{N}$ and $\delta''\colon V''\rightarrow \mathbb{N}$ such that $(V',\delta')$ and $(V'',\delta'')$ are minimal weighted vertex covers for $G_{\omega}$. Since $|(V',\delta')|\neq  |(V'',\delta'')|$ we conclude that $G_{\omega}$ is mixed.
\end{proof}

\section{Monomial Ideals}
\label{sec120507b}

In this section, we include some background material on monomial ideals in
the polynomial ring
$R=A[X_1,\ldots,X_d]$.

\begin{Definition}\label{Def monomial set}
A \emph{monomial} in $R$ is an element of the form $X_1^{a_1}\cdots X_d^{a_d}$
where the $a_i$ are non-negative integers.
A \textbf{monomial ideal} in $R$ is an ideal generated by a (possibly empty)
set of monomials of $R$.
For each monomial ideal $I\subset R$, let $\llbracket I \rrbracket$ denote the set of all monomials contained in $I$.
\end{Definition}

\begin{Definition}\label{Def Ideal generated by vertex cover}
For each subset $V'\subseteq V$, let $P(V')\subseteq R$ be the ideal ``generated by the elements of $V'$'':
$$P(V')=(X_i\mid v_i\in V')R.$$
For each subset $V'\subseteq V$ and for each function $\delta' \colon V'\rightarrow \mathbb{N}$, let $P(V', \delta ')\subseteq R$ be the ideal ``generated by the elements of $(V',\delta')$'': 
\[
P(V', \delta ') = (X_i^{\delta'(v_i)}\mid v_i \in V')R.
\]
We say that 
the ideals $P(V',\delta')$ are \textbf{m-irreducible},
to indicate that they are irreducible with respect to 
intersections of monomial ideals.
\end{Definition}

\begin{Note}
The notation $V'=\{v_{i_1}^{\delta'(v_{i_1})},\dots,v_{i_k}^{\delta'(v_{i_k})}\}$
is handy because it essentially lists the generators of $P(V', \delta ')$.
\end{Note}

\begin{Ex}\label{Ex P}
The ideals $P(V',\delta')$ coming from the three weighted vertex covers
in Example~\ref{ex120507a} are
$(X_1^2,X_2^5,X_4^3,X_5^2)R$,
$(X_1^2,X_2^5,X_4^2)R$,
and
$(X_1^2,X_2^5,X_4^3)R$.
Notice that the ideal corresponding to the minimal weighted vertex cover is contained 
in the ideals corresponding to non-minimal weighted vertex covers.
\end{Ex}

\begin{Ex}
We have $P(\emptyset)=(\emptyset)R=0$ and 
$P(\emptyset,\delta')=(\emptyset)R=0$.
\end{Ex}

\begin{Note}\label{Note 1}
A monomial ideal $I\subseteq R$ is of the form $P(V',\delta')$ if and only if it is generated by ``pure powers'' of the variables, that is, by monomials of the form $X_i^{e_i}$. 
When $A$ is a field, 
the ideals $P(V')$ are precisely the prime monomial
ideals, and
the ideals $P(V',\delta')$ are precisely the irreducible monomial
ideals. 
\end{Note}

\begin{Definition} \label{Def m-height}
Given an ordered pair $(V',\delta')$ the  \textbf{m-height} of $P(V',\delta')$ is 
$$\mht(P(V',\delta'))=|V'|.$$
Given a monomial ideal $I\subset R$ such that $I = \bigcap_{i=1}^m P(V_i,\delta_i)$, the \textbf{m-height} of $I$ is
$$\mht(I) = \min_i\{\mht(P(V_i,\delta_i))\}.$$
\end{Definition}

\begin{Note}
Assume that  $A$ is a field.
In this case, each ideal $P(V')$ is prime in $R$,
and $\mht(P(V',\delta'))=\mht(P(V'))$ is the same as
$\Ht(P(V',\delta'))=\Ht(P(V'))$.
We use the notation $\mht$ in general to indicate that we are taking the
height with respect to monomial prime ideals.
\end{Note}

\begin{Definition}\label{Def graphs m-unmixed}
Assume that $I = \bigcap_{i=1}^m P(V_i,\delta_i)$ is an irredundant m-irreducible decomposition, that is, such that there are no
containment relations between the ideals in the intersection. Then $I$ is \textbf{m-unmixed} provided that $\mht(P(V_i,\delta_i)) =\mht(P(V_j,\delta_j))$ for all $i,j$, that is, if $\mht(P(V_i,\delta_i))=\mht(I)$ for all $i$. 
We say that $I$ is \textbf{m-mixed} if it is not m-unmixed.
\end{Definition}

\begin{Note}
If $A$ is a field, then a monomial ideal $I\subset R$ is m-unmixed if and only if
it is unmixed.
\end{Note}

\section{Weighted Edge Ideals and Their Decompositions}
\label{sec120508a}

In this section, we define the edge ideal of a weighted graph 
and establish some of its fundamental properties.
Recall that $G$ is a graph with vertex set $V$ and edge set $E$, and
$R=A[X_1,\ldots,X_d]$.

\begin{convention*}
In this section, $G_{\omega}$ is a weighted graph.
\end{convention*}

\begin{Definition} \label{Def edge ideal}
The \textbf{edge ideal} associated to $G$ is the ideal $I(G) \subseteq R$ that is ``generated by the edges of G'':
\[
I(G) = (X_iX_j\mid v_iv_j \in E)R.
\]
The \textbf{weighted edge ideal} associated to $G_\omega$ is the ideal $I(G_{\omega})\subseteq R$ that is ``generated by the weighted edges of G'':
\[
I(G_{\omega})= (X_i^{\omega(e)}X_j^{\omega(e)}\mid e= v_iv_j \in E)R.
\]
\end{Definition}

\begin{Note}\label{note120508b}
M-irreducible decompositions for the edge ideal $I(G)$
are $I(G)=\bigcap_{V'}P(V')=\bigcap_{\min V'}P(V')$
where the first intersection is taken over the set of all vertex covers of $G$,
and the second intersection is taken over the set of all minimal vertex covers of $G$;
see, e.g., \cite[Theorem 5.3.9]{SSW}.
Furthermore, the  second decomposition is irredundant. 
One of the points of this section is to provide analogous decompositions
for $I(G_{\omega})$. This is done in Theorem~\ref{Cor I(G)= int.P}.
\end{Note}

The following lemma is the first key to decomposing the edge ideal of
$G_{\omega}$.

\begin{Lemma} \label{lem vertex cover containment}
Given subsets $V', V'' \subseteq V$ and functions $\delta'\colon  V'\rightarrow \mathbb{N}, \delta''\colon V''\rightarrow \mathbb{N}$, we have $(V'', \delta '') \leq (V', \delta ')$ if and only if $P(V'', \delta'')\subseteq P(V', \delta')$.
\end{Lemma}

\begin{proof}
Let us begin by assuming that $(V'',\delta'')\leq(V',\delta')$. Then we have 
$V''\subseteq V'$ and $\delta'(v_i)\leq\delta''(v_i)$ for all $v_i\in V''$. To show that $P(V'',\delta'')\subseteq P(V',\delta')$ we need to show that each generator $X_i^{\delta''(v_i)}$ of $P(V'',\delta'')$ is in $P(V',\delta')$. By assumption, we have $v_i\in V''\subseteq V'$ and $\delta''(v_i)\geq \delta'(v_i)$. Thus, 
the condition $X_i^{\delta'(v_i)}\in P(V',\delta')$  implies that $X_i^{\delta''(v_i)}= X_i^{\delta''(v_i)-\delta'(v_i)} X_i^{\delta'(v_i)}\in P(V',\delta')$. Hence $P(V'',\delta'')\subseteq P(V',\delta ')$.

For the converse assume that $P(V'',\delta'')\subseteq P(V',\delta')$. Then $X_i^{\del''(v_i)}\in P(V',\del')$ for all $v_i\in V''$. Therefore, there is a generator $X_j^{\delta'(v_j)}\in P(v',\delta')$ such that $X_j^{\delta'(v_j)}|X_i^{\delta'(v_i)}$. Since $\delta'(v_j)\geq 1$, it follows that $i=j$ and $\delta'(v_j)\leq \delta''(v_i)$. Thus, $v_i=v_j\in V'$ and $\delta'(v_i)=\delta'(v_j)\leq \delta''(v_j)$.
Since this is so
for all $v_i\in V''$, we have  $(V'',\del'')\leq(V',\del')$, by definition.
\end{proof} 

The next result is the second key to decomposing $I(G_{\omega})$.

\begin{Lemma} \label{Thrm weighted edge ideal containment}
Given a subset $V'\subseteq V$ and a function $\delta'\colon V'\to\mathbb N$,
one has
$I(G_{\omega}) \subseteq P(V',\delta ')$ if and only if $(V', \delta ')$ is a weighted vertex cover of $G_{\omega}$.
\end{Lemma}

\begin{proof}
Write $V'=\{v_{i_1},\dots, v_{i_k}\}$. Assume first that $I(G_{\omega}) \subseteq P(V',\delta ')$. Then for all $e=v_iv_j\in E$, we have $X_i^{\omega(e)}X_j^{\omega(e)} \in I(G_{\omega})\subseteq P(V', \delta ')=(X_{i_1}^{\delta ' (v_{i_1})}, \dots,X_{i_k}^{\delta ' (v_{i_k})})$. Thus $X_{i_\ell}^{\delta ' (v_{i_\ell})} | X_i^{\omega(e)}X_j^{\omega(e)}$ for some $1 \leq \ell \leq k$. Since $\delta'(v_{i_{\ell}})\geq 1$, we conclude that 
either $i_{\ell} = i$ and $\delta ' (v_{i_{\ell}}) \leq \omega(e)$,
or $i_{\ell} = j$ and $\delta ' (v_{i_{\ell}}) \leq \omega(e)$. Thus, 
either $v_i=v_{i_\ell}\in V'$ and $\delta ' (v_{i_{\ell}}) \leq \omega(e)$,
or $v_j=v_{i_\ell}\in V'$ and $\delta ' (v_{i_{\ell}}) \leq \omega(e)$.
Since this is so for each edge in $G$, we 
conclude that $(V',\delta')$ is a weighted vertex cover of $G_\omega$. 

For the converse assume that $(V', \delta')$ is a weighted vertex cover of $G_\omega$. We need to show that each generator of $I(G_{\omega})$ is an element of $P(V', \delta')$. Let $X_i^{\omega(e)}X_j^{\omega(e)}$ be a generator of $I(G_{\omega})$ corresponding to the edge $e=v_iv_j$ with weight $\omega(e)$ in $G_{\omega}$. Since $(V', \delta ')$ is a weighted vertex cover, we have two cases. The first case is when $v_i \in V'$ and $\delta ' (v_i) \leq \omega(e)$; in this case, we have ${X_i}^{\delta ' (i)} | X_i^{\omega(e)}X_j^{\omega(e)}$ and so $X_i^{\omega(e)}X_j^{\omega(e)} \in P(V', \delta ')$. Similarly, if $v_j\in V'$ and $\del'(v_j)\leq\omega(e)$, then $X_i^{\omega(e)}X_j^{\omega(e)} \in P(V', \delta ')$. Thus $I(G_{\omega}) \subseteq P(V',\delta ')$. 
\end{proof}

Here is our decomposition result for $I(G_{\omega})$.

\begin{Theorem} \label{Cor I(G)= int.P}
Let $G_{\omega}$ be a weighted graph with vertex set $V =\{v_1,\ldots, v_d\}$. Then 
$$I(G_{\omega})= \bigcap_{(V',\delta')}P(V',\delta')= \bigcap_{\text{min } (V',\delta')}P(V',\delta')$$
where the first intersection is taken over all weighted vertex covers of $G_{\omega}$ and the second intersection is taken over all minimal weighted vertex covers of $G_{\omega}$. Furthermore, the second decomposition is irredundant.
\end{Theorem}

\begin{proof}
Every monomial ideal can be written as a finite (possibly empty) intersection of 
m-irreducible ideals, i.e., 
ideals of the form $P(V',\delta')$; see, e.g., \cite[Theorems 4.1.4 and 4.3.1]{SSW}.
This implies that $I(G_{\omega})$ is a finite intersection of ideals of the form $P(V',\delta')$, and Lemma~\ref{Thrm weighted edge ideal containment} implies that the only $(V',\delta')$ that can occur in such a decomposition are weighted vertex covers for $G_{\omega}$. Thus, we have $I(G_\omega)= \bigcap_{(V',\delta')}P(V',\delta')$.

Since every minimal weighted vertex cover is a weighted vertex cover we have 
$$\bigcap_{(V',\delta')}P(V',\delta')\subseteq \bigcap_{\text{min } (V',\delta')}P(V',\delta').$$ 
The reverse containment 
$$\bigcap_{(V',\delta')}P(V',\delta')\supseteq \bigcap_{\text{min } (V',\delta')}P(V',\delta')$$ 
follows from Proposition \ref{Prop mwvc contained in wvc} and Lemma \ref{lem vertex cover containment}.

Finally, the intersection $\bigcap_{\text{min } (V',\delta')}P(V',\delta')$ is irredundant 
by Lemma~\ref{lem vertex cover containment}.
\end{proof} 

Theorem~\ref{Cor I(G)= int.P} proves the next 
result that connects
unmixedness for graphs and edge ideals.

\begin{Corollary}\label{cor120508a}
The graph $G$ is unmixed if and only if the ideal $I(G)$ is m-unmixed. 
The weighted graph $G_{\omega}$ is unmixed if and only if the ideal $I(G_{\omega})$ is m-unmixed. 
\end{Corollary}

\begin{Remark}\label{rmk120610a}
Corollary~\ref{cor120508a} shows that m-unmixedness of $I(G_{\omega})$ is independent
of the ring $A$ since it only depends on the unmixedness of $G_{\omega}$. 
\end{Remark}

\begin{Ex}\label{ex120508a}
We decompose $I(P^2_{\omega})$ where $P^2_{\omega}$ is the following
weighted 2-path:
$$\xymatrix{
v_1 \ar@{-}[r]^-a
&v_2 \ar@{-}[r]^-b
&v_3.}$$
Assume by symmetry that $a\leq b$.
In this case, we have
\begin{align*}
I(P^2_{\omega})
&=(X_1^aX_2^a,X_2^bX_3^b)R\\
&=(X_1^a,X_2^bX_3^b)R\bigcap (X_2^a,X_2^bX_3^b)R\\
&=(X_1^a,X_2^b)R\bigcap(X_1^a,X_3^b)R\bigcap (X_2^a)R.
\end{align*}
If $a<b$, then this decomposition is irredundant.
By Lemma~\ref{lem vertex cover containment}, we conclude that
there are exactly three minimal weighted vertex covers for 
$P^2_{\omega}$, namely
$\{v_1^a,v_2^b\}$, $\{v_1^a,v_3^b\}$ and $\{v_2^a\}$.

On the other hand, if $a=b$, then we have
$(X_2^a)R=(X_2^b)R\subseteq (X_1^a,X_2^b)R$,
and hence
\begin{align*}
I(P^2_{\omega})
=(X_1^a,X_3^b)R\bigcap (X_2^a)R.
\end{align*}
We deduce that there are exactly two minimal weighted vertex covers for 
$P^2_{\omega}$ in this case.

In either case, we conclude that $P^2_{\omega}$ is mixed and
$I(P^2_{\omega})$ is m-mixed. See Section~\ref{sec120507d}
for more information about weighted paths.
\end{Ex}

\begin{Ex}\label{ex120508b}
We decompose $I(C^3_{\omega})$ where $C^3_{\omega}$ is the following
weighted 3-cycle:
$$\xymatrix{
v_1 \ar@{-}[r]^-a\ar@{-}[rd]_-c
&v_2 \ar@{-}[d]^-b\\
&v_3.}$$
Assume by symmetry that $a\leq b\leq c$.
In this case, we decompose as in Example~\ref{ex120508a} to find
\begin{align*}
I(C^3_{\omega})
&=(X_1^aX_2^a,X_2^bX_3^b,X_1^cX_3^c)R\\
&=(X_1^a,X_2^b)R
\bigcap(X_1^a,X_3^b)R
\bigcap(X_1^c,X_2^a)R
\bigcap(X_2^a,X_3^c)R.
\end{align*}
It follows that $C^3_{\omega}$ is unmixed and
$I(C^3_{\omega})$ is m-unmixed. 
It is worth noting that, when $a<b<c$, this decomposition is irredundant
with two ideals of the form $(X_1^p,X_2^q)R$;
this sort of behavior does not occur in the unweighted case.
See Sections~\ref{sec120507c}--\ref{sec120507d}
for more information about weighted cycles.
\end{Ex}

We end this section with a few results about associated primes and
(un)mixedness. 

\begin{Definition}\label{Def m-rad}
Let $I$ be a monomial ideal in $R$. The \textbf{monomial radical} of $I$ is the monomial ideal $\mrad(I)=(S)R$ where 
$$S=\llbracket R\rrbracket\bigcap\rad(I)=\{z\in\llbracket R\rrbracket \mid z^n\in I \text{ for some } n\geq 1\}$$
where $\rad(I)$ is the radical of $I$ and $\llbracket R\rrbracket$ is from
Definition~\ref{Def monomial set}.
\end{Definition}

\begin{Note}\label{note120508a}
If $A$ is a field (more generally, if $A$ is reduced), then 
$\mrad(I)=\rad(I)$ for each monomial ideal $I\subseteq R$.
\end{Note}

\begin{Lemma} \label{Prop m-rad of specific ideals}
We have
$\mrad(I(G_{\omega}))=I(G)$ and
$\mrad(P(V',\delta'))=P(V')$
for each ordered pair $(V',\delta')$.
\end{Lemma}

\begin{proof}
Given a monomial $f=X_{i_1}^{a_{i_1}}\cdots X_{i_n}^{a_{i_n}}$ where
each $a_i\geq 1$, set
$\operatorname{red}(f)=X_{i_1}\cdots X_{i_n}$.
If $I$ is generated by the set $S\subseteq\llbracket R\rrbracket$,
then $\mrad(I)$ is generated by the set 
$\{\operatorname{red}(f)\mid f\in S\}$; see, e.g., \cite[Proposition 3.5.5]{SSW}.
The desired conclusions now follow.
\end{proof}

\begin{Proposition}\label{Prop primes and vertex covers}
Assume that $A$ is an integral domain. 
\begin{enumerate}[\rm(a)]
\item \label{item120508a1}
The minimal primes of $I(G_{\omega})$ are the ideals $P(V')$ such that $V'$ is a minimal vertex cover of $G$.
\item  \label{item120508a2}
The associated primes of $I(G_{\omega})$ are the ideals $P(V')$ such that $(V',\delta')$ is a minimal weighted vertex cover of $G_{\omega}$.
\end{enumerate}
\end{Proposition}

\begin{proof}
\eqref{item120508a1} The minimal primes of $I(G_{\omega})$ are the m-irreducible components of $\rad(I(G_{\omega}))=\mrad(I(G_{\omega}))=I(G_{\omega}))$ by 
Note~\ref{note120508a} and Lemma~\ref{Prop m-rad of specific ideals}. 
From Note~\ref{note120508b} we know that 
$I(G)=\bigcap_{\min V'}P(V')$ is an irredundent irreducible decomposition
where the intersection is take over the set of all minimal vertex covers of $G$.
It follows that the minimal primes of $I(G_{\omega})$ are the ideals $P(V')$ such that $V'$ is a minimal vertex cover of $G$, as claimed.

\eqref{item120508a2} The associated primes of $I(G_{\omega})$ are the radicals of the  m-irredundant irreducible components of $I(G_{\omega})$. 
Theorem~\ref{Cor I(G)= int.P} implies that $I(G_{\omega})=\bigcap_{\min(V',\delta')}P(V',\delta')$ is an irredundant m-irreducible decomposition 
where the intersection is take over the set of all minimal weighted vertex covers of $G_{\omega}$.
Hence, the associated primes of $I(G_{\omega})$ are the ideals $\rad(P(V',\delta'))=P(V')$ where $(V',\delta')$ is a minimal weighted vertex cover of $G_{\omega}$.
\end{proof}

\begin{Proposition}\label{Prop trivial G_w unmixed iff G unmixed}
A trivially weighted graph $G_{\omega}$ is unmixed if and only if $G$ is unmixed.
\end{Proposition}

\begin{proof}
The forward implication is from Proposition \ref{Prop mixed ei then mixed wei}.

For the converse assume that $G$ is unmixed. Since $G_{\omega}$ is trivially weighted, let the weight of the each edge in $G_{\omega}$ be $a$. 
Given a monomial ideal $I\subseteq R$, 
set $I^{[a]}=(\{f^a\mid f\in \llbracket I \rrbracket\})R$
where the notation $\llbracket I \rrbracket$ is from
Definition~\ref{Def monomial set}.
Since $G_{\omega}$ is trivially weighted, it is straightforward to show that
$I(G_\omega)=I(G)^{[a]}$.
Furthermore, given the m-irreducible decomposition
$I(G)=\bigcap_{\min V'}P(V')$ from Note~\ref{note120508b},
we have
$$I(G_\omega)=I(G)^{[a]}
=\bigcap_{\min V'}P(V')^{[a]}=\bigcap_{(V',\delta')}P(V',\delta')$$ 
where $\delta'(v_i):=a$ for all $v_i\in V'$;
see, e.g., \cite[Proposition 7.1.3]{SSW}.
Since $G$ is unmixed, each $V'$ in this intersection has the same cardinality. 
It follows that each $(V',\delta')$ has the same cardinality.
Therefore  $G_\omega$ is unmixed.
\end{proof}

\section{Weighted Cycles and Weighted Complete Graphs}
\label{sec120507c}

In this section, we characterize the weighted cycles and complete 
graphs  that are unmixed.

\begin{Fact} \label{Thrm cycles mixedness}
 $C^n$ is unmixed if and only if $n\in \{3,4,5,7\}$;
see, e.g., \cite[Exercise 6.2.15]{Vi}. 
\end{Fact}

We treat the weighted cycles case-by-case.
Here is a convenient summary of these results.

\begin{enumerate}[(a)]
\item If $C^d_{\omega}$ is unmixed, then $d\in\{3,4,5,7\}$
by Proposition~\ref{Prop mixed ei then mixed wei} and
Fact~\ref{Thrm cycles mixedness}.
\item Every $C^3_{\omega}$ is unmixed by Example~\ref{ex120508b}.
\item  $C^4_{\omega}$ is unmixed if and only if it is trivially weighted
by Propositions~\ref{Prop weighted 4-cycles mixed} and~\ref{Prop trivially weighted 4-cycle}.
\item $C^5_{\omega}$: see Theorem~\ref{Prop weighted 5-cycles mixedness}.
\item  $C^7_{\omega}$ is unmixed if and only if it is trivially weighted
by Propositions~\ref{Prop trivially weighted 4-cycle} and~\ref{Prop weighted 7-cycle mixed}.
\end{enumerate}

\begin{Proposition}\label{Prop trivially weighted 4-cycle}
For $n\in\{3,4,5,7\}$, every trivially weighted $n$-cycle $C^n_{\omega}$ is unmixed.
\end{Proposition}

\begin{proof}
From Fact~\ref{Thrm cycles mixedness} we know that $C^n$ is unmixed. Thus,  Proposition \ref{Prop trivial G_w unmixed iff G unmixed} implies that $C^n_{\omega}$ is unmixed.
\end{proof}

\begin{Proposition} \label{Prop weighted 4-cycles mixed}
Every nontrivially weighted 4-cycle, $C^4_{\omega}$ is mixed.
\end{Proposition}

\begin{proof}
Let us consider a non-trivially weighted 4-cycle whose underlying unweighted graph is $C^4 = v_1v_2v_3v_4v_1$ and the weights of the edges are as follows: 
$$
\xymatrix{
v_1 \ar@{-}[d]_{d} \ar@{-}[r]^{a} &v_2\ar@{-}[d]^{b}\\
v_4 \ar@{-}[r]_{c} & v_3
}
$$
By symmetry, assume that $a$ is the smallest weight on any edge. Then since $C^4$ is not trivially weighted, at least one edge has weight strictly greater then $a$. By symmetry assume that $a <  b$. Now we demonstrate two minimal vertex covers of different cardinalities. 
First, we consider $V'=\{v_2^a,v_4^{min(c,d)}\}$. 
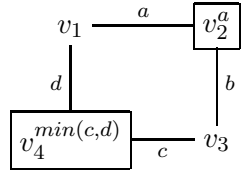
$$
\xymatrix{
v_1 \ar@{-}[d]_{d} \ar@{-}[r]^{a} &*+[F]{v_2^a}\ar@{-}[d]^{b}\\
*+[F]{v_4^{min(c,d)}} \ar@{-}[r]_-{c} & v_3
}
$$
Notice that since $a < b$, the edges $v_1v_2$ and $v_2v_3$ are covered by $v_2^a$ and since $\min(c,d) \leq c,d$, the edges $v_3v_4$ and $v_4v_1$ are covered. The removal of either of these vertices would not result in a vertex cover. If we increase the weight on the vertex $v_2$, then the edge $v_1v_2$ will not be covered; and if we increase the weight on the vertex $v_4$, then the edge with the smaller weight (either $v_3v_4$ or $v_4v_1$) would not be covered. Thus, $V'$ is a minimal weighted vertex cover with cardinality 2.

Now let $V''= \{v_1^a,v_2^b,v_4^{c}\}$.
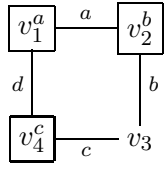
$$
\xymatrix{
*+[F]{v_1^a} \ar@{-}[d]_{d} \ar@{-}[r]^{a} &*+[F]{v_2^b}\ar@{-}[d]^{b}\\
*+[F]{v_4^{c}} \ar@{-}[r]_-{c} & v_3
}
$$
Notice that the vertex $v_1^a$ covers the edges $v_1v_2$ and $v_1v_4$, the vertex $v_2^b$ covers the edge $v_2v_3$ and the vertex $v_4^{c}$ covers the edge $v_3v_4$. Furthermore, if we remove $v_1^a$ from $V''$ or increase the weight, the edge $v_1v_2$ is not covered. If we remove the vertex $v_2^b$ from $V''$ or increase the weight, the edge $v_2v_3$ is not covered. If we remove the vertex $v_4^{c}$ from $V''$ or increase the weight then the edge $v_3v_4$ is not covered. Hence $V''$ is a minimal vertex cover with cardinality 3.
Thus $C^4_{\omega}$ is mixed.
\end{proof}

\begin{Theorem}\label{Prop weighted 5-cycles mixedness}
Let $C^5_{\omega}$ be a weighted 5-cycle whose underlying unweighted graph is $C^5= v_1v_2v_3v_4v_5v_1$.
Then $C^5_{\omega}$ is unmixed if and only if 
it is isomorphic to the weighted 5-cycle
\begin{equation}\label{eq120509a}
\begin{split}
\xymatrix{
v_1 \ar@{-}[d]_e \ar@{-}[r]^a &v_2\ar@{-}[r]^b &v_3 \ar@{-}[ld]^c\\
v_5\ar@{-}[r]_d & v_4.
}
\end{split}
\end{equation}
such that $e=a\leq b\geq c\leq d\geq e$.
\end{Theorem}

\begin{proof}
Let us first assume that $C^5_{\omega}$ is isomorphic to the weighted 
5-cycle~\eqref{eq120509a} and that
$e=a\leq b\geq c\leq d\geq e$; we show that $C^5_{\omega}$ is unmixed. In order to do this we decompose the edge ideal of  $C^5_{\omega}$ as in Example~\ref{ex120508a}. 
\begin{align*}
I( C^5_{\omega}) 
&= (X_1^aX_2^a,X_2^bX_3^b,X_3^cX_4^c,X_4^dX_5^d,X_5^eX_1^e)\\
&=(X_1^a,X_2^b,X_4^c)
\bigcap(X_1^a,X_3^c,X_4^d)\bigcap(X_1^a,X_3^c,X_5^d)
\bigcap(X_1^a,X_3^b,X_4^c)\\ 
&\hspace{5mm}\bigcap (X_2^a,X_3^c,X_5^e)
\bigcap (X_2^a,X_4^c,X_5^e)
\bigcap (X_2^a,X_4^c,X_1^e)
\end{align*}
Therefore $C^5_\omega$ is unmixed when the weight on the edges are as specified.

For the converse we will assume that 
the weighted 5-cycle~\eqref{eq120509a}
is unmixed. By  Proposition \ref{Prop mvc implies mwvc}
and Fact~\ref{Thrm cycles mixedness}, every minimal weighted vertex cover
of $C^5_\omega$ has cardinality 3. We  proceed by steps to eliminate all possible cases of the comparability of the weights of the edges of $C^5_\omega$ besides our hypothesized conclusion. In each step
we derive contradictions by building minimal weighted vertex covers for that have cardinality greater than 3.

\textbf{Step 1.} Let us first suppose that $e< a< b$. We  consider two cases. 

Case 1:  $e\geq d< c$. 
We claim $V'=\{v_1^a,v_2^b,v_4^d,v_5^e\}$ is a minimal weighted vertex cover. 
$$
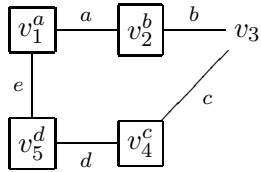
\xymatrix{
*+[F]{v_1^{a}} \ar@{-}[d]_e \ar@{-}[r]^a &*+[F]{v_2^b}\ar@{-}[r]^b &v_3 \ar@{-}[ld]^c\\
*+[F]{v_5^d} \ar@{-}[r]_d & *+[F]{v_4^c}
}$$
It is routine to verify that all the weighted edges are covered by $V'$. We verify that $V'$ is a minimal weighted vertex cover. Notice if we remove the weighted vertex $v_1^a$ or increase the weight, then the edge $v_1v_2$ is not covered. If we remove the weighted vertex $v_2^b$ or increase the weight, then the edge $v_2v_3$ is not covered. If we remove the weighted vertex $v_4^c$ or increase the weight, then the edge $v_3v_4$ is not covered.  If we remove the weighted vertex $v_5^d$ or increase the weight, then the edge $v_4v_5$ is not covered. Thus $V'$ is a minimal weighted vertex cover of cardinality 4, contradicting the unmixedness of $C^5_{\omega}$.
 
Case 2: either $e\leq d$ or $e\geq d \geq c$. 
We claim that $V''=\{v_1^a,v_2^b,v_4^c,v_5^e\}$ is a minimal weighted vertex cover. 
$$
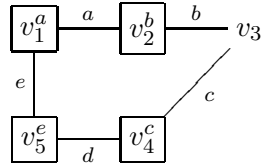
\xymatrix{
*+[F]{v_1^{a}} \ar@{-}[d]_e \ar@{-}[r]^a &*+[F]{v_2^b}\ar@{-}[r]^b &v_3 \ar@{-}[ld]^c\\
*+[F]{v_5^e} \ar@{-}[r]_d & *+[F]{v_4^c}
}$$
As in Case 1, it follows readily that $V''$ is a minimal weighted vertex cover of cardinality 4. 

\textbf{Step 2.}
Let us next suppose that no two adjacent edges have the same weight. Then since the cycle is of odd length we conclude by symmetry that 
$C^5_{\omega}$ is isomorphic to a graph~\eqref{eq120509a} such that
$e< a< b$. Step 1 shows that this contradicts the unmixedness of $C^5_{\omega}$.
We conclude that there are two adjacent edges with the same weight.

\textbf{Step 3.} By symmetry, we assume that $e=a$. Now,  suppose that
 $b<a$ and $d<a$. We again consider two cases.
If $d>c\leq b$, then
it is readily shown that $V'=\{v_2^a,v_3^c,v_4^d,v_5^a\}$ is a minimal weighted vertex cover. 
$$
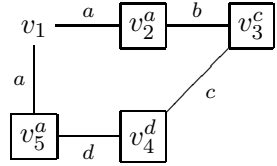
\xymatrix{
v_1 \ar@{-}[d]_a \ar@{-}[r]^a &*+[F]{v_2^a}\ar@{-}[r]^b &*+[F]{v_3^c} \ar@{-}[ld]^c\\
*+[F]{v_5^a} \ar@{-}[r]_d & *+[F]{v_4^d}
}$$ 
On the other hand, if  $d\geq c\geq b$, $d\leq c\leq b$, or $d\leq c\geq b$,
then  $V''=\{v_2^a,v_3^b,v_4^d,v_5^a\}$ is a minimal weighted vertex cover.
$$
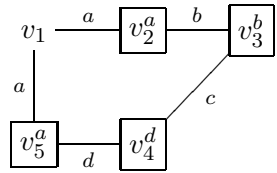
\xymatrix{
v_1 \ar@{-}[d]_a \ar@{-}[r]^a &*+[F]{v_2^a}\ar@{-}[r]^b &*+[F]{v_3^b} \ar@{-}[ld]^c\\
*+[F]{v_5^a} \ar@{-}[r]_d & *+[F]{v_4^d}
}$$ 
In each case we have exhibited a minimal weighted vertex cover of cardinality 4, which is a contradiction. Therefore, either $a\leq b$ or $a\leq d$.

\textbf{Step 4.} By symmetry, assume that $a\leq b$.
Suppose that $b<c$. We consider six cases.

Case 1:  $a<b<c$. 
In this case, $V'=\{v_1^a,v_2^b,v_3^c,v_5^d\}$ is a minimal weighted vertex cover.
$$
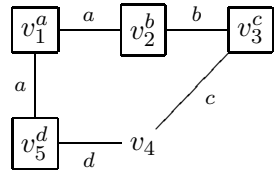
\xymatrix{
*+[F]{v_1^a} \ar@{-}[d]_a \ar@{-}[r]^a &*+[F]{v_2^b}\ar@{-}[r]^b &*+[F]{v_3^c} \ar@{-}[ld]^c\\
*+[F]{v_5^d} \ar@{-}[r]_d &v_4
}$$ 

Case 2:  $a=b<c<d$.  
Here, $V''=\{v_1^a,v_2^b,v_3^c,v_4^d\}$ is a minimal weighted vertex cover.
$$
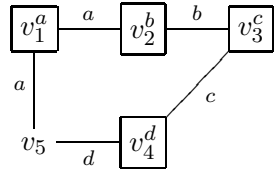
\xymatrix{
*+[F]{v_1^a} \ar@{-}[d]_a \ar@{-}[r]^a &*+[F]{v_2^b}\ar@{-}[r]^b &*+[F]{v_3^c} \ar@{-}[ld]^c\\
v_5 \ar@{-}[r]_d &*+[F]{v_4^d}
}$$ 

Case 3:  $a=b<c>d>a$.  
In this case,  $V'''=\{v_1^a,v_2^b,v_4^c,v_5^d\}$ is a minimal weighted vertex cover.
$$
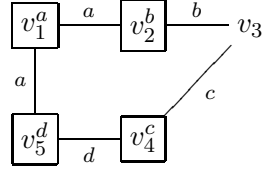
\xymatrix{
*+[F]{v_1^a} \ar@{-}[d]_a \ar@{-}[r]^a &*+[F]{v_2^b}\ar@{-}[r]^b &v_3 \ar@{-}[ld]^c\\
*+[F]{v_5^d} \ar@{-}[r]_d &*+[F]{v_4^c}
}$$ 

Case 4:  $a=b<c\geq d\leq a$.
This case fits our hypothesized conclusion.

Case 5: $a=b<c=d\leq a$.
This case is not possible because it states that $a<c\leq a$.

Case 6: $a=b<c=d>a$.
This case is covered by Step 3.

\textbf{Step 5.} Assume that $a\leq b \geq c$ and suppose $c>d$. We consider three cases.

 Case 1:  $a<b\geq c>d$.
Here, $V'''=\{v_1^a,v_2^b,v_4^c,v_5^d\}$ is a minimal weighted vertex cover.
$$
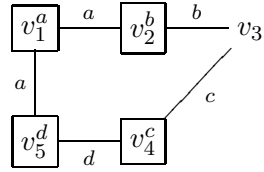
\xymatrix{
*+[F]{v_1^a} \ar@{-}[d]_a \ar@{-}[r]^a &*+[F]{v_2^b}\ar@{-}[r]^b &v_3 \ar@{-}[ld]^c\\
*+[F]{v_5^d} \ar@{-}[r]_d &*+[F]{v_4^c}
}$$ 

 Case 2:  $a=b=c>d<a$  or $a=b=c>d\leq a$.
This case fits our desired conclusion.

 Case 3:  $a=b\geq c>d\geq a$.
This case is not possible because it states that $a\geq c> a$.

\textbf{Step 6.} Assume that $a\leq b \geq c\leq d$ and suppose 
that $a>d$.
If $c<d$, then $V'=\{v_1^a,v_2^b,v_4^c,v_5^a\}$ is a minimal weighted vertex cover. 
$$
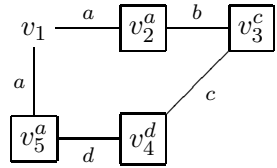
\xymatrix{
v_1 \ar@{-}[d]_a \ar@{-}[r]^a &*+[F]{v_2^a}\ar@{-}[r]^b &*+[F]{v_3^c} \ar@{-}[ld]^c\\
*+[F]{v_5^a} \ar@{-}[r]_d & *+[F]{v_4^d}
}$$
On the other hand, if $c=d$, then we have
$c=d<e=a\leq b\geq c$ which fits our  conclusion.

Thus, if $C^5_\omega$ is unmixed, we have  $e=a\leq b \geq c\leq d\geq e$,
as claimed. 
\end{proof}

\begin{Proposition} \label{Prop weighted 7-cycle mixed}
Every nontrivially weighted 7-cycle is mixed. 
\end{Proposition}

\begin{proof}
Let us consider a weighted 7-cycle whose underlying unweighted graph is $C^7 = v_1v_2v_3v_4v_5v_6v_7v_1$,  weighted as follows: 
$$
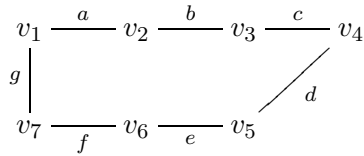
\xymatrix{
v_1 \ar@{-}[d]_g \ar@{-}[r]^a &v_2\ar@{-}[r]^b &v_3 \ar@{-}[r]^c &v_4 \ar@{-}[ld]^d\\
v_7 \ar@{-}[r]_f  &v_6\ar@{-}[r]_e &v_5
}
$$
By symmetry, let us say that $a$ is the smallest weight on any edge. Then since $C^7$ is not trivially weighted, at least one edge has weight strictly greater then $a$. By symmetry, assume that $a < g$. 

Since the unweighted
$C^7$ is unmixed and each minimal  vertex cover has cardinality 4,   Proposition~\ref{Prop mvc implies mwvc} provides a 
minimal weighted vertex cover of $C^7_\omega$ with cardinality 4. 
Now we will demonstrate a weighted vertex cover $V''$ such that $|V''|=5$. 
We consider two cases.

Case 1: $f\geq e\leq d$. In this case, $V''=\{v_1^g,v_2^a,v_3^c,v_5^d,v_6^e\}$
 is a minimal weighted vertex cover of $C^7_\omega$.
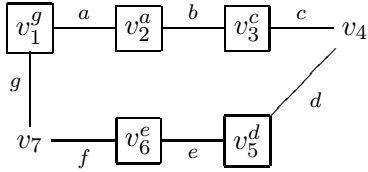
$$
\xymatrix{
*+[F]{v_1^g} \ar@{-}[d]_g \ar@{-}[r]^a & *+[F]{v_2^a}\ar@{-}[r]^b &*+[F]{v_3^c} \ar@{-}[r]^c &v_4 \ar@{-}[ld]^d\\
v_7 \ar@{-}[r]_f &*+[F]{v_6^e}\ar@{-}[r]_e &*+[F]{v_5^d}
}
$$

Case 2: $f\leq e \geq d$, $f\geq e\geq d$ or $f\leq e\leq d$. In this case, 
$V''=\{v_1^g,v_2^a,v_3^c,v_5^d,v_6^f\}$
is a minimal weighted vertex cover of $C^7_\omega$.
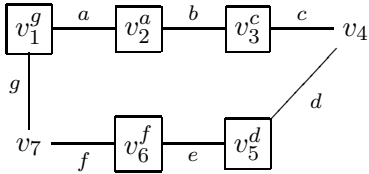
$$
\xymatrix{
*+[F]{v_1^g} \ar@{-}[d]_g \ar@{-}[r]^a & *+[F]{v_2^a}\ar@{-}[r]^b &*+[F]{v_3^c} \ar@{-}[r]^c &v_4 \ar@{-}[ld]^d\\
v_7 \ar@{-}[r]_f &*+[F]{v_6^f}\ar@{-}[r]_e &*+[F]{v_5^d}
}
$$
Since we have demonstrated 2 minimal weighted vertex covers for $C^7_\omega$ of different cardinalities we conclude that our nontrivially weighted $C^7_\omega$ is mixed.
\end{proof}

\begin{Proposition}\label{Prop complete graphs unmixed}
Every weighted complete graph $K^n_{\omega}$ is unmixed.
\end{Proposition}

\begin{proof} 
It is easily verified that the smallest minimal vertex cover for $K^n$ is of cardinality $n-1$. Therefore by Proposition \ref{Prop mvc implies mwvc} the smallest minimal weighted vertex cover for $K^n_{\omega}$  also has cardinality $n-1$. We show that there is not a minimal weighted vertex cover of cardinality $n$. Assume that $(V',\delta')$ is a minimal weighted vertex cover with cardinality $n$. By symmetry assume the vertex $v_1$ has the maximal weight of all the vertices of $V'$. Now consider the removal of $v_1$ from $V'$. Since $\delta'(v_1)$ was maximal and all other vertices of $V$ are in $(V',\delta')$ then every edge adjacent to $v_1$ must be covered by the other vertex adjacent to that edge. Thus $V'$ is not minimal and every minimal weighted vertex cover has cardinality $n-1$ which implies $K^n_{\omega}$ is unmixed.
\end{proof}

\section{Cohen-Macaulay Weighted Graphs}
\label{sec120507d}

In this section, we prove Theorems~\ref{intthm120507a}
and~\ref{intthm120507b} from the introduction.

\begin{convention*}
In this section, $A$ is a field.
\end{convention*}

\begin{Definition}
The weighted graph $G_{\omega}$ is 
\textbf{Cohen-Macaulay over $A$} if the ring $R/I(G_\omega)$ is Cohen-Macaulay.
If $G_\omega$ is Cohen-Macaulay over every field, we simply say that it is
\textbf{Cohen-Macaulay}.
\end{Definition}

The Cohen-Macaulay
weighted complete graphs are easily identified.

\begin{Proposition}\label{Prop complete graphs CM}
Every weighted complete graph $K^n_{\omega}$ is Cohen-Macaulay.
\end{Proposition}

\begin{proof} 
By Proposition \ref{Prop complete graphs unmixed} we know that $K^n_{\omega}$ is unmixed. Since $A$ is a field, we also know that $\dim (K^n_{\omega}) = 1$ because the cardinality of the minimal vertex covers are $n-1$. Since unmixed in dimension 1 implies Cohen-Macaulay, we conclude that $K^n_{\omega}$ is Cohen-Macaulay. 
\end{proof}

Next, we characterize the Cohen-Macaulay weighted suspensions and trees.
One main point is the following lemma
whose proof  is essentially due to J.~Herzog; see~\cite[Proposition 2.2]{Vi1} 
and~\cite[Proposition 6.3.2]{Vi}.

\begin{Lemma}\label{Thrm conditions for R/I CM}
Let $S= A[Y_1,\dots, Y_n,Z_1,\dots, Z_n]$ be a polynomial ring over $A$, and
fix a subset  $M \subseteq \{(i,j)\mid 1\leq i<j\leq n\}$. Then the ideal 
$$I = (Y_i^{a_{i}}Z_i^{a_{i}}, Z_i^{b_{ij}}Z_j^{b_{ij}}\mid \text{$i=1,\dots, n$ and 
$(i,j)\in M$ and $a_i\geq b_{ij}\leq a_j$} )S$$ 
is such that $S/I$ is Cohen-Macaulay.
\end{Lemma}

\begin{proof}
We  polarize the ideal $I$ to obtain 
\begin{align*}
S'&= k[Y_{1,1},\dots, Y_{1,a_1},\dots,Y_{n,1}\dots, Y_{n,a_n},Z_{1,1}\dots, Z_{1,a_1},\dots, Z_{n,1}\dots, Z_{n,a_n}] \\
I'&=(Y_{1,1}\cdots Y_{1,a_1}Z_{1,1}\cdots Z_{1,a_1},\dots, Y_{n,1}\cdots Y_{n,a_n}Z_{n,1}\cdots Z_{n,a_n}, Z_{i,1}\cdots Z_{i,b_{ij}}Z_{j,1}\cdots Z_{j,b_{ij}})S'.
\end{align*}
By general properties of polarization, the next sequence is
$S'$-regular and $S'/I'$-regular:
\begin{align*}
&Y_{1,1}-Y_{1,2},\dots, Y_{1,1}-Y_{1,a_1},\dots, Y_{n,1}-Y_{n,2},\dots, Y_{n,1}-Y_{n,a_n},\\
&Z_{1,1}-Z_{1,2},\dots, Z_{1,1}-Z_{1,a_1},\dots,Z_{n,1}-Z_{n,2},\dots, Z_{n,1}-Z_{n,a_n}.
\end{align*} 

Note that $I'$ is a polarization of
the ideal $J=(Z_i^{2a_i},Z_i^{b_{i,j}}Z_j^{b_{i,j}})T$ where
$T= A[Z_1,\dots, Z_n]$.
The ring $T/J$ is Artinian, so it is Cohen-Macaulay. 
Since $T/J$ is obtained from $S'/I'$ by modding out by a homogeneous regular sequence, it follows that $S'/I'$ is Cohen-Macaulay.
Similarly, we conclude  that $S/I$ is Cohen-Macaulay, as desired. 
\end{proof}

\begin{Definition}\label{def120607a}
Recall that $G$ has vertex set $V(G)=\{v_1,\ldots,v_d\}$.
A \emph{suspension} of $G$ is a graph $H$ whose vertex set is
$V(H)=V(G)\cup\{w_1,\ldots,w_d\}$ and whose edge set is
$E(H)=E(G)\cup\{v_1w_1,\ldots,v_dw_d\}$.
In other words, $H$ is obtained from $G$ by adding to $G$ a new vertex $w_i$ and edge
(sometimes called a ``whisker'') $v_iw_i$ for each vertex $v_i$ of $G$.
\end{Definition}

\begin{Note} \label{note120607a}
Let $H$ be a suspension of $G$.
Graphically, this says that $T$ has the form
$$
\xymatrix{
&w_k\ar@{-}[d] &w_i\ar@{-}[d] &w_j\ar@{-}[d]\\
 \ar@{.}[r]&v_k\ar@{-}[r] &v_i\ar@{-}[r] &v_j\ar@{.}[r]&
}
$$ 
where the bottom ``row'' is the graph $G$.
(Note that this sketch is deceptively oversimplified since the bottom row 
can be any graph.)
\end{Note}

\begin{Definition}\label{def120607b}
Let $G_{\omega}$ and $H_{\lambda}$ be  weighted graphs.
Then $H_{\lambda}$ is a \emph{weighted suspension} of $G_{\omega}$
if the underlying graph $H$ is a suspension of $G$ and
(with notation as in Definition~\ref{def120607a})
we have $\lambda(v_iv_j)=\omega(v_iv_j)$ for all $v_iv_j\in E(G)$.
\end{Definition}

\begin{Theorem}\label{thm120607}
Let $H_{\lambda}$ be a weighted suspension of a weighted graph $G_{\omega}$,
with notation as in Definition~\ref{def120607a}.
Then the following conditions are equivalent:
\begin{enumerate}[\rm(i)]
\item \label{thm120607a} $H_{\lambda}$ is Cohen-Macaulay,
\item \label{thm120607b} $H_{\lambda}$ is unmixed,
\item \label{thm120607c} for each
$v_iv_j\in E(G)$ we have $\lambda(v_iv_j) \leq \lambda(v_iw_i)$ and $\lambda(v_iv_j) \leq \lambda(w_jv_j)$.
\end{enumerate}
\end{Theorem}

\begin{proof}
$\eqref{thm120607a}\implies\eqref{thm120607b}$
This is standard.

$\eqref{thm120607b}\implies\eqref{thm120607c}$
Assume  that $H_{\lambda}$ is unmixed.
Since  the underlying unweighted graph $H$ is a suspension, 
it is Cohen-Macaulay by~\cite[Proposition 2.2]{Vi1}. 
In particular, it is unmixed.
It is straightforward to show that the set $V'=E(G)$ is a minimal vertex cover for $H$,
so each minimal vertex cover of $H$ has cardinality $d$.
Proposition~\ref{Prop mvc implies mwvc} implies that the cardinality of each minimal weighted vertex cover of $H_{\lambda}$ is also $d$. 
Suppose that there exists some $i$ such that $a=\lambda(w_iv_i)<\lambda(v_iv_j)=b$ and $c=\lambda(w_jv_j)$. We derive a contradiction by constructing a minimal weighted vertex cover $V''$ such that 
$|V''|= d+1$. 

For each $k\neq i,j$ set 
$$e_k=\min\{\lambda(w_kv_k),\lambda(v_kv_l)\mid v_kv_l\in E(H_{\lambda})\}.$$
Let $V'=\{v_i^b,w_i^a,w_j^c,v_k^{e_k}\mid k\neq i,j\}$. 
It is routine to verify that this is indeed a weighted vertex cover of $H_{\lambda}$. 
Proposition~\ref{Prop mwvc contained in wvc} implies that there is a minimal weighted vertex cover $(V'',\delta'')$ 
of $G_{\lambda}$ such that $(V'',\del'')\leq(V',\del')$.
Note that for $k\neq i,j$, the vertex $v_k$ cannot be removed from $V'$ since this would
leave the edge $w_kv_k$ uncovered. (However, it may be that the weight on $v_k$ can
be increased.) The vertex $w_j$ cannot be removed from $V'$, and its weight cannot be increased,
because this would leave the edge $w_jv_j$ uncovered.
If we remove $v_i^b$ from $V'$ or increase the weight, then the edge $v_iv_j$ is not covered. If we remove $w_i^a$ or increase the weight, then the edge $w_iv_i$ is not covered. Thus $V''$ is a minimal weighted vertex cover such that $|V''|=|V'|=r+1$,
providing the desired contradiction. 

$\eqref{thm120607c}\implies\eqref{thm120607a}$
This follows from Lemma~\ref{Thrm conditions for R/I CM}.
\end{proof}

\begin{Remark}\label{rmk120610b}
Remark~\ref{rmk120610b}
shows that the equivalence of conditions~\eqref{thm120607b} and~\eqref{thm120607c} of
Theorem~\ref{thm120607} do not need the assumption that $A$ is a field.
Similar comments hold for 
Theorem~\ref{Thrm conditions for trees iff CM}
and
Corollary~\ref{Cor path results}.
\end{Remark}

\begin{Note} \label{note120511a}
Cohen-Macaulay unweighted trees have been explicitly characterized  as follows:
a tree $T$ is Cohen-Macaulay if and only if
either $|V(T_{\omega})|\leq 2$, or 
$T$ is a suspension of a tree;
see,
e.g.,
\cite[Theorem 6.3.4 and Corollary 6.3.5]{Vi}.
We see next that a weighted tree is Cohen-Macaulay if and only if
its  underlying unweighted graph has this form, with some restrictions on the weights.
\end{Note}

Theorem~\ref{intthm120507b} from the introduction is a consequence of the next result.

\begin{Theorem}\label{Thrm conditions for trees iff CM}
Let $H_{\lambda}$ be a weighted tree.
Then the following conditions are equivalent:
\begin{enumerate}[\rm(i)]
\item \label{item120509x} $H_{\lambda}$ is Cohen-Macaulay,
\item \label{item120509y} $H_{\lambda}$ is unmixed,
\item \label{item120509z} one of the following holds:
\begin{enumerate}[\rm(1)]
\item \label{item120509a} $|V(H_{\lambda})|\leq 2$ or 
\item \label{item120509b}  $H_{\lambda}$ is a weighted suspension of a weighted tree
$G_{\omega}$ such that (with notation as in Definition~\ref{def120607a})
we have $\lambda(v_iv_j) \leq \lambda(v_iw_i)$ and $\lambda(v_iv_j) \leq \lambda(w_jv_j)$
for each
$v_iv_j\in E(G)$.
\end{enumerate}
\end{enumerate}
In particular, if $H_{\lambda}$ is Cohen-Macaulay is Cohen-Macaulay, then so is $H$.
\end{Theorem}

\begin{proof} 
$\eqref{item120509x}\implies\eqref{item120509y}$
This is standard.

$\eqref{item120509y}\implies\eqref{item120509z}$
Assume  that $H_{\lambda}$ is unmixed and that $|V(H_{\lambda})|> 2$;
we need to show that condition~(2) is satisfied.
By Proposition \ref{Prop mixed ei then mixed wei} the underlying unweighted graph $H$ is unmixed. Since we have $|V(H_{\lambda})|> 2$, it follows from 
Note~\ref{note120511a} that $H$ is a suspension of a tree $G$.
It follows readily that $H_{\lambda}$ is a weighted suspension of a weighted tree
$G_{\omega}$. The condition
$\lambda(v_iv_j) \leq \lambda(v_iw_i)$ and $\lambda(v_iv_j) \leq \lambda(w_jv_j)$
for each
$v_iv_j\in E(G)$ follows from Theorem~\ref{thm120607}.

$\eqref{item120509z}\implies\eqref{item120509x}$
This follows from Lemma~\ref{Thrm conditions for R/I CM} or Theorem~\ref{thm120607}.
\end{proof}

For instance, Theorem~\ref{Thrm conditions for trees iff CM} provides 
the following explicit
characterization of Cohen-Macaulay weighted paths.

\begin{Corollary}\label{Cor path results}
Let $P_{\omega}$ be a weighted path.
Then the following conditions are equivalent:
\begin{enumerate}[\rm(i)]
\item \label{item120509x1} $P_{\omega}$ is Cohen-Macaulay,
\item \label{item120509y2} $P_{\omega}$ is unmixed,
\item \label{item120509z3} one of the following holds:
$P_{\omega}$ is of length 1 or of length 3 of the following form
$$\xymatrix{
x_1 \ar@{-}[r]^a &y_1\ar@{-}[r]^b &y_2 \ar@{-}[r]^c&x_2
}$$
such that $b\leq a$ and $b\leq c$.
\end{enumerate}
\end{Corollary}

The following examples are useful for the proof of 
Proposition~\ref{Prop tw7-c not cm}.

\begin{Example}\label{ex120512a}
Let $P^4_\omega$ be a trivially weighted 4-path where each edge has
weight $a$.
$$\xymatrix{v_1\ar@{-}[r]^a
&v_2\ar@{-}[r]^a&v_3\ar@{-}[r]^a&v_4\ar@{-}[r]^a&v_5}$$
We show that $R/I(P^4_\omega)$ has
dimension 3, depth 2, and type 1.

As in Example~\ref{ex120508a}, we decompose:
\begin{align*}
I(P^4_\omega)
&=(X_1^aX_2^a,X_2^aX_3^a,X_3^aX_4^a,X_4^aX_5^a)R \\
&=
(X_1^a,X_3^a,X_4^a)R\bigcap
(X_1^a,X_3^a,X_5^a)R\bigcap
(X_2^a,X_3^a,X_5^a)R\bigcap
(X_2^a,X_4^a)R.
\end{align*}
It follows that $\dim(R/I(P^4_\omega))=3$

Using the above decomposition,
we conclude that the associated prime ideals of $I(P^4_\omega)$
are 
$(X_1,X_3,X_4)R$,
$(X_1,X_3,X_5)R$,
$(X_2,X_3,X_5)R$, and
$(X_2,X_4)R$.
In particular, the element $X_4-X_5$ is $R/I(P^4_\omega)$-regular.
We simplify the quotient
\begin{align*}
R/(I(P^4_\omega)+(X_4-X_5)R)
&\cong R'/(X_1^aX_2^a,X_2^aX_3^a,X_3^aX_4^a,X_4^{2a})R'
\end{align*}
where $R'=A[X_1,X_2,X_3,X_4]$.
As before, we decompose:
\begin{align*}
(X_1^aX_2^a,X_2^aX_3^a,X_3^aX_4^a,X_4^{2a})R'
&=
(X_1^a,X_3^a,X_4^{2a})R'\bigcap
(X_2^a,X_3^a,X_4^{2a})R'\bigcap
(X_2^a,X_4^a)R'.
\end{align*}
The associated primes of this ideal are
$(X_1,X_3,X_4)R'$,
$(X_2,X_3,X_4)R'$, and
$(X_2,X_4)R'$.
It follows that the element $X_1-X_2$ is 
$R/(I(P^4_\omega)+(X_4-X_5)R)$-regular, so we have
$\depth(R/I(P^4_\omega))\geq 2$, as claimed.
We simplify the quotient
\begin{align}
\label{eq120514b}
R/(I(P^4_\omega)+(X_4-X_5,X_1-X_2)R)
&\cong R''/(X_2^{2a},X_2^aX_3^a,X_3^aX_4^a,X_4^{2a})R''
\end{align}
where $R''=A[X_2,X_3,X_4]$ and decompose:
\begin{align}
(X_2^{2a},X_2^aX_3^a,X_3^aX_4^a,X_4^{2a})R''
&=
\label{eq120514a}
(X_2^{2a},X_3^a,X_4^{2a})R''\bigcap
(X_2^a,X_4^a)R''.
\end{align}
Since the maximal ideal $(X_2,X_3,X_4)R''$ is associated to 
$(X_2^{2a},X_2^aX_3^a,X_3^aX_4^a,X_4^{2a})R''$, 
this shows that $\depth(R/I(P^4_\omega))= 2$.
Furthermore, this explains the non-vanishing in the next computation:
\begin{align*}
0
&\neq\operatorname{Ext}_R^2(R/(X_1,\ldots,X_5)R,R/I(P^4_\omega))\\
&\cong\operatorname{Hom}_{R''}(R''/(X_2,X_3,X_4)R'',R''/(X_2^{2a},X_2^aX_3^a,X_3^aX_4^a,X_4^{2a})R'')\\
&\cong((X_2^{2a},X_2^aX_3^a,X_3^aX_4^a,X_4^{2a})R'':_{R''}(X_2,X_3,X_4)R'')\\
&=((X_2^{2a},X_3^a,X_4^{2a})R''\bigcap
(X_2^a,X_4^a)R'':_{R''}(X_2,X_3,X_4)R'')\\
&\subseteq\left((X_2^{2a},X_3^a,X_4^{2a})R'':_{R''}(X_2,X_3,X_4)R''\right)\\
&=(X_2^{2a-1}X_3^{a-1}X_4^{2a-1})R''.
\end{align*}
The first isomorphism is standard from the fact that 
$X_4-X_5,X_1-X_2$ is $R$-regular and 
$R/I(P^4_\omega)$-regular with the isomorphism~\eqref{eq120514b}.
The second isomorphism and the containment are routine.
The first equality comes from the decomposition~\eqref{eq120514a},
and
the second equality is from the fact that $A$ is a field.
It follows that
$\operatorname{Ext}_R^2(R/(X_1,\ldots,X_5)R,R/I(P^4_\omega))$
is cyclic, so $R/I(P^4_\omega)$ has type 1, as claimed.
\end{Example}

\begin{Example}\label{ex120513a}
Let $P^5_\omega$ be a trivially weighted 5-path where each edge has
weight $a$.
$$\xymatrix{v_1\ar@{-}[r]^a
&v_2\ar@{-}[r]^a&v_3\ar@{-}[r]^a&v_4\ar@{-}[r]^a&v_5\ar@{-}[r]^a&v_6}$$
As in Example~\ref{ex120512a}, the quotient $R/I(P^4_\omega)$ has
dimension 3, depth 2, and type 1.
\end{Example}

Now we turn our attention to Cohen-Macaulayness of weighted cycles.

\begin{Proposition} \label{Prop 3-cycles cm}
Every weighted 3-cycle $C^3_{\omega}$ is Cohen-Macaulay.
\end{Proposition}

\begin{proof}
From the decomposition  of $I(C^3_{\omega})$ in Example~\ref{ex120508b}, we see 
that $I(C^3_{\omega}))$ is m-unmixed; since 
$R/I(C^3_{\omega})$ has dimension 1,  it is Cohen-Macaulay. 
\end{proof}

\begin{Proposition}\label{Prop tw4-c not cm}
No weighted 4-cycle is Cohen-Macaulay.
\end{Proposition}

\begin{proof}
Let $C^4_\omega$ be a weighted 4-cycle.
If $C^4_\omega$ is non-trivially weighted, then it is mixed
by Proposition~\ref{Prop weighted 4-cycles mixed}, hence it is not Cohen-Macaulay.
Thus, we assume that $C^4_\omega$ is trivially weighted.
Write the underlying unweighted graph of $C^4_{\omega}$ as $C^4=v_1v_2v_3v_4v_1$, and let the weight of each edge of $C^4_{\omega}$ be $a$. Then $I(C^4_{\omega})=(X_1^aX_2^a,X_2^aX_3^a,X_3^aX_4^a,X_4^aX_1^a)$. 
Decomposing $I(C^4_{\omega})$ and computing associated primes as in
Example~\ref{ex120512a},
we see  that 
$X_1-X_2$ is a regular element for $R/I(C^4_{\omega})$ such that
$$R'=R/(I(C^4_{\omega})+(X_1-X_2)R)\cong
A[X_1,X_3,X_4]/(X_1^{2a},X_1^aX_3^a,X_3^aX_4^a,X_4^aX_1^a).$$ 
Also, as in Example~\ref{ex120512a}, the maximal ideal of $R'$ is associated to $R'$. 
It follows that $R/I(C^4_{\omega})$ has depth 1 and dimension 2,
so $C^4_{\omega}$ is not Cohen-Macaulay.
\end{proof}

\begin{Theorem}\label{Prop w5-c cm iff um}
A weighted 5-cycle $C^5_{\omega}$ is Cohen-Macaulay if and only if it is unmixed.
\end{Theorem}

\begin{proof}
One implication is standard. 
For the converse, assume that $C^5_{\omega}$ is unmixed.
Theorem~\ref{Prop weighted 5-cycles mixedness} 
implies that $C^5_{\omega}$
is isomorphic to the weighted 5-cycle
$$\xymatrix{
v_1 \ar@{-}[d]_e \ar@{-}[r]^a &v_2\ar@{-}[r]^b &v_3 \ar@{-}[ld]^c\\
v_5\ar@{-}[r]_d & v_4
}$$
such that $e=a\leq b\geq c\leq d\geq e$.
Partially decomposing the edge ideal of $C^5_{\omega}$  we obtain:
\begin{align*}
I( C^5_{\omega}) &= (X_1^aX_2^a,X_2^bX_3^b,X_3^cX_4^c,X_4^dX_5^d,X_5^eX_1^e)
=J\bigcap K
\end{align*}
where $J=(X_1^aX_2^a,X_3^c,X_4^dX_5^d,X_5^eX_1^e)$
and $K=(X_1^aX_2^a,X_2^bX_3^b,X_4^c,X_5^eX_1^e)$.
It is straightforward to show that these ideals fit into  an exact sequence of the following form: 
$$0 \to \frac{R}{I(C^5_{\omega})}\to
\frac{R}{J} \oplus \frac{R}{K}\to
\frac{R}{(X_3^c,X_4^c,X_1^aX_2^a,X_5^eX_1^e)}\to 0. $$
The quotient $R/(X_3^c,X_4^c,X_1^aX_2^a,X_5^eX_1^e)$ 
has depth 1 and dimension 2,
because it can be obtained from the ring $A[X_1,X_2,X_5]/(X_1^aX_2^a,X_5^eX_1^e)$
which has depth 1 and dimension 2. 

Furthermore, Corollary~\ref{Cor path results} implies that 
$A[X_1,X_2,X_4,X_5]/(X_1^aX_2^a,X_4^dX_5^d,X_5^eX_1^e)$
and $A[X_1,X_2,X_3,X_5]/(X_1^aX_2^a,X_2^bX_3^b,X_5^eX_1^e)$
are Cohen-Macaulay of dimension 2.
Hence, $R/J$ and  $R/K$ are Cohen-Macaulay of depth 2. 
Thus by the Depth Lemma, $R/I(C^5_{\omega})$ has depth at least 2. 
Since it has dimension 2, it is Cohen-Macaulay. 
\end{proof}

\begin{Proposition}\label{Prop tw7-c not cm}
No weighted 7-cycle is  Cohen-Macaulay.
\end{Proposition}

\begin{proof}
Let $C^7_\omega$ be a weighted 7-cycle.
If $C^7_\omega$ is non-trivially weighted, then it is mixed
by Proposition~\ref{Prop weighted 7-cycle mixed}, hence it is not Cohen-Macaulay.
Thus, we assume that $C^7_\omega$ is trivially weighted.
Write the underlying unweighted graph of $C^7_{\omega}$ as 
$C^7=v_1v_2v_3v_4v_5v_6v_7v_1$, and let the weight of each edge of 
$C^7_{\omega}$ be $a$. 
Partially decomposing the edge ideal of $C^7_{\omega}$  we obtain:
\begin{align*}
I(C^7_{\omega})&=(X_1^aX_2^a,X_2^aX_3^a,X_3^aX_4^a,X_4^aX_5^a,X_5^aX_6^a,X_6^aX_7^a,X_7^aX_1^a)
=J\bigcap K
\intertext{where} 
J&=(X_1^a,X_2^aX_3^a,X_3^aX_4^a,X_4^aX_5^a,X_5^aX_6^a,X_6^aX_7^a)\\
K&=(X_2^a,X_3^aX_4^a,X_4^aX_5^a,X_5^aX_6^a,X_6^aX_7^a,X_7^aX_1^a).
\end{align*}
It is routine to show that these ideals fit into  an exact sequence of the following form: 
\begin{equation}\label{eq120514c}
0\to\frac{R}{I(C^7_{\omega})}\to \frac{R}{J} \oplus \frac{R}{K}\to
\frac{R}{L}\to 0
\end{equation}
where
$$L=(X_1^a,X_2^a,X_3^aX_4^a,X_4^aX_5^a,X_5^aX_6^a,X_6^aX_7^a)R.$$
Example~\ref{ex120512a} implies that
the ring
$A[X_3,X_4,X_5,X_6,X_7]/(X_3^aX_4^a,X_4^aX_5^a,X_5^aX_6^a,X_6^aX_7^a)$
has depth 2 and type 1,
and it follows that 
$R/L$
also has depth 2 and type 1.
Similarly, Example~\ref{ex120513a} implies that
$R/J$ and $R/K$ both have depth 2 and type 1.
The Depth Lemma applied to the sequence~\eqref{eq120514c}
implies that
$\depth(R/I(C^7_{\omega}))\geq 2$.
Furthermore, for the ideal $\mathfrak m=(X_1,\ldots,X_7)R$,
part of the long exact sequence in
$\operatorname{Ext}_R^2(R/\mathfrak m,-)$
associated to the sequence~\eqref{eq120514c}
has the form 
$$0\to 
\operatorname{Ext}_R^2(R/\mathfrak m,R/I(C^7_\omega))\to
\operatorname{Ext}_R^2(R/\mathfrak m,R/J)
\oplus 
\operatorname{Ext}_R^2(R/\mathfrak m,R/K)
\to \operatorname{Ext}_R^2(R/\mathfrak m,R/L).$$
Using the type computations we have already made, this sequence has
the form
$$0\to 
\operatorname{Ext}_R^2(R/\mathfrak m,R/I(C^7_\omega))\to
k^2\to k.$$
It follows that $\operatorname{Ext}_R^2(R/\mathfrak m,R/I(C^7_\omega))\neq 0$, 
so  $\depth(R/I(C^7_\omega))=2<3=\dim(R/I(C^7_\omega))$.
It follows that $C^7_\omega$ is not Cohen-Macaulay, as claimed.
\end{proof}

\begin{proof}[Proof of Theorem~\ref{intthm120507a}]
\eqref{intthm120507a1}
Assume that $C^n_{\omega}$ is Cohen-Macaulay.
Then it is  unmixed, so Proposition~\ref{Prop mixed ei then mixed wei}
implies that the unweighted cycle $C^n$ is unmixed.
From Fact~\ref{Thrm cycles mixedness}, we conclude that
$n\in\{3,4,5,7\}$. Propositions~\ref{Prop tw4-c not cm}
and~\ref{Prop tw7-c not cm} imply that $n\neq 4,7$, so we have
$n\in\{3,5\}$.

\eqref{intthm120507a2}
This is Proposition~\ref{Prop 3-cycles cm}.

\eqref{intthm120507a3}
Theorems~\ref{Prop weighted 5-cycles mixedness} and~\ref{Prop w5-c cm iff um}.
\end{proof}

\bibliographystyle{amsplain}
\providecommand{\bysame}{\leavevmode\hbox to3em{\hrulefill}\thinspace}
\providecommand{\MR}{\relax\ifhmode\unskip\space\fi MR }
\providecommand{\MRhref}[2]{%
  \href{http://www.ams.org/mathscinet-getitem?mr=#1}{#2}
}
\providecommand{\href}[2]{#2}

\end{document}